\documentclass[11pt,dvipsnames]{article}
\usepackage{amsmath,amsfonts,amssymb,amsthm,graphicx,verbatim,subfig,mathtools,wasysym}

\usepackage[usenames]{xcolor}
\usepackage[all]{xy}
\usepackage[a4paper,left=25mm]{geometry}
\usepackage[colorlinks=true, linkcolor=black, citecolor=black]{hyperref}

\definecolor{darkred}{RGB}{200,0,0}
\definecolor{zurz}{RGB}{180,20,80}

\ifx\pdftexversion\undefined

\usepackage[a4paper,colorlinks,linkcolor=black,filecolor=black,citecolor=black,urlcolor=black,pdfstartview=FitH]{hyperref}
\fi

\font\sixbb=msbm6
\font\eightbb=msbm8
\font\twelvebb=msbm10 scaled 1095
\newfam\bbfam
\textfont\bbfam=\twelvebb \scriptfont\bbfam=\eightbb
                           \scriptscriptfont\bbfam=\sixbb
\def\bb{\fam\bbfam\twelvebb}
\newcommand{\Rea}{\mathbb{R}}

\newcommand{\Int}{{\bb Z}}

\newcommand{\FF}{{\bb F}}

\newcommand{\PPR}{\mathbb{P}}

\DeclareMathAlphabet\EuR{U}{eur}{m}{n}
\SetMathAlphabet\EuR{bold}{U}{eur}{b}{n}

\newtheorem{theorem}{Theorem}[section]

\newtheorem{conjecture}[theorem]{Conjecture}
\newtheorem{proposition}[theorem]{Proposition}
\newtheorem{corollary}[theorem]{Corollary}
\newtheorem{lemma}[theorem]{Lemma}
\newtheorem{definition}[theorem]{Definition}
\newtheorem{example}[theorem]{Example}

\newcommand{\enp}{\begin{flushright} $\Box$ \end{flushright}}
\newcommand{\beq}[0]{\begin{equation}}
\newcommand{\enq}[0]{\end{equation}}
\newcommand{\dn}{\Delta_{n-1}}
\newcommand{\rk}{{\rm rank}}
\newcommand{\thh}{\tilde{H}}

\newcommand{\pr}{{\rm Pr}}

\newcommand{\lk}{{\rm lk}}

\newcommand{\inj}{\text{Inj}}
\newcommand{\sd}{\mathrm{sd}\,}

\newcommand{\ccs}{\mathbb{S}}
\newcommand{\tilb}{\tilde{\beta}}
\newcommand{\hd}{{\rm hd}}

\title{On the topology of complexes of injective words}
\author{Wojtek Chacholski\thanks{Matematik, KTH, 10054 Stockholm, Sweden, email: wojtek@kth.se~. Supported by VR grant 2014-04770.}
\and Ran Levi\thanks{Institute of Mathematics, University of Aberdeen, Aberdeen, UK, email: r.levi@abdn.ac.uk~. Supported by EPSRC grant EP/P025072/1}
\and Roy Meshulam \thanks{Department of Mathematics,
Technion, Haifa 32000, Israel. e-mail:
meshulam@technion.ac.il~. Supported by ISF grant 326/16.}}

\begin{document}
\maketitle

\begin{abstract}
An injective word over a finite alphabet $V$ is a sequence $w=v_1v_2\cdots v_t$ of distinct elements of $V$. The set $\inj(V)$ of injective words on $V$ is partially ordered by inclusion. A complex of injective words is the order complex $\Delta(W)$ of a subposet $W \subset \inj(V)$. Complexes of injective words arose recently in applications of algebraic topology to neuroscience, and are of independent interest in topology and combinatorics. In this article we mainly study Permutation Complexes, i.e. complexes of injective words $\Delta(W)$, where $W$ is the downward closed subposet of  $\inj(V)$ generated by a set of permutations of $V$. In particular, we determine the homotopy type of $\Delta(W)$ when $W$ is generated by two permutations, and prove that any stable homotopy type is realizable by a permutation complex. We describe a homotopy decomposition for the complex of injective words $\Gamma(K)$ associated with a simplicial complex $K$, and point out a connection to a result of Randal-Williams and Wahl. Finally, we discuss some probabilistic aspects of random permutation complexes.
\end{abstract}

\section{Introduction}
Numerous ideas from combinatorial topology recently found applications in science and technology. A particular example occurred in studying the Blue Brain Project reconstruction of the so called neocortical column of a 14 days old rat \cite{Megapaper}. In this model, as well as in a biological brain, one frequently finds neurons that are reciprocally connected. In \cite{RN-16} the authors consider the directed graphs that emerge from the structure and function of the model, and to each such graph they associate the so called ``directed flag complex".  This is a topological object made out of the directed cliques in the graph. The topological properties of the resulting spaces turn out to reveal interesting information about the reconstruction. This paper arose from considering the possible homotopy types of these directed flag complexes.

We proceed with some formal definitions.
Let $V$ be a finite set. An {\it injective word} $w$ on $V$ of length $|w|=k$ is an ordered $k$-tuple $v_1\cdots v_k$ such that $v_i \in V$ and $v_i \neq v_j$ for $1 \leq i \neq j \leq k$. Let $\inj(V)$ denote the set of nonempty injective words on $V$, partially ordered by inclusion, i.e.
$w=v_1 \cdots v_k \leq w'=v_1' \cdots v_{\ell}'$ if there exist $1 \leq t_1<\cdots<t_k \leq  \ell$ such that
$v_j=v_{t_j}'$ for $1 \leq j \leq k$.  The {\it order complex} of a poset
$(P,\leq)$, denoted by $\Delta(P)$, is the simplicial complex on the vertex set $P$, whose $k$-simplices are the chains $x_0< \cdots <x_k$ of $P$. For $P=[n]=\{1,\ldots,n\}$ with its usual ordering, $\Delta([n])=\Delta_{n-1}$ is the standard $(n-1)$-simplex.
\begin{definition}
\label{d:osc}
A \emph{complex of injective words} is an order complex $\Delta(W)$ associated with a set of injective words $W \subset \inj(V)$.
\end{definition}

Complexes of injective words play a role in a number of areas, ranging from topological combinatorics
(see e.g. \cite{BW83}) to homological stability of groups (see e.g. \cite{RW-W17}).
Recently low dimensional examples of these complexes were discovered while studying a digital reconstruction of brain tissue of a rat  \cite{RN-16}.
Let $\ccs_n$ denote the symmetric group on $[n]$.
For $n \geq 1$ let $D(n)=\sum_{j=0}^n (-1)^j \frac{n!}{j!}$ be the number of derangements (i.e. fixed point free permutations) in $\ccs_n$. Farmer \cite{Farmer78} proved that if $|V|=n$, then $\Delta(\inj(V))$ has the homology of a wedge of $D(n)$ copies of the $(n-1)$-sphere $S^{n-1}$. The following homotopical strengthening of Farmer's theorem was obtained by Bj\"{o}rner and Wachs \cite{BW83}.
\begin{theorem}[\cite{BW83}]
\label{t:bw}
$$\Delta(\inj([n])) \simeq \bigvee_{D(n)} S^{n-1}.$$
\end{theorem}
\noindent
In this paper we study several aspects of complexes of injective words.
For $w_1,\ldots,w_m \in \inj(W)$ let
$$W(w_1,\ldots,w_m)=\{w \in \inj(V):w \leq w_i \text{~for~some~} 1 \leq i \leq m\}.$$
The \emph{complex of injective words generated by $w_1,\ldots,w_m$} is $$X(w_1,\ldots,w_m)=\Delta(W(w_1,\ldots,w_m)).$$
For $p \geq 0$ let $A_p(w_1,\ldots,w_m)$ denote the free abelian group generated by all
words $w \in W(w_1,\ldots,w_m)$ of length $p+1$. Let $\partial_p:A_p(w_1,\ldots,w_m) \rightarrow A_{p-1}(w_1,\ldots,w_m)$ be the $p$-boundary map given by
$\partial_p(v_0\cdots v_p)=\sum_{j=0}^p (-1)^j (v_0\cdots v_{j-1}v_{j+1} \cdots v_p)$.
Farmer \cite{Farmer78} proved the following
\begin{theorem}[\cite{Farmer78}]
\label{t:acom}
The homology of $X(w_1,\ldots,w_m)$ is isomorphic to the homology of the chain complex
$\bigoplus_{p \geq 0} A_p(w_1,\ldots,w_m)$.
\end{theorem}
\noindent
A permutation $\sigma \in \ccs_n$ will be identified with the injective word
$\sigma(1)\cdots\sigma(n) \in \inj([n])$.
\begin{definition}
\label{d:uop}
A \emph{permutation complex} on $V=[n]$ is a complex of the form
$X(\sigma_1,\ldots,\sigma_m)$, where $\sigma_1,\ldots,\sigma_m \in \ccs_n$.
\end{definition}
\noindent
Any simplicial complex $Y$ on the vertex set $[n]$ is homeomorphic to a
subcomplex of $\Delta(\inj([n]))$.
Indeed, let $(P(Y),\prec)$ denote the poset of nonempty simplices of $Y$ ordered by inclusion. The {\it barycentric subdivision} $\sd Y$ of $Y$, is the order complex $\Delta(P(Y))$.
Associate with each nonempty simplex $\sigma \in Y$, the word $w_{\sigma}$ that lists the vertices of $\sigma$ in the natural order on $[n]$. Then the map $\sigma \rightarrow w_{\sigma}$ is a simplicial isomorphism between $\sd Y$ and $\Delta(\{w_{\sigma}: \sigma \in Y\})$. On the other hand, homeomorphism types of permutation complexes are considerably more restricted. For example, the claw graph $G$ with edges $12,13,14$ is not homeomorphic to any permutation complex.
Our first result shows however that the stable homotopy type of any finite simplicial complex is realizable by a permutation complex. Let $\Sigma Y=S^0 * Y$ denote the (unreduced) suspension of a space $Y$. Note that
$\Sigma \emptyset =S^0$.
\begin{theorem}
\label{t:uniform}
For any finite simplicial complex $Y$ there exist positive $r,n,m$ and premutations
$\sigma_1,\ldots,\sigma_m \in \ccs_n$ such that
$$\Sigma^r Y  \simeq X(\sigma_1,\ldots \sigma_m).$$
\end{theorem}

The complex $X(\sigma)$ on a single permutation $\sigma \in \ccs_n$ is clearly isomorphic to the barycentric subdivision of the $(n-1)$-simplex $\Delta_{n-1}$, and hence contractible. Our second result is a characterization of the homotopy type
of complexes generated by two permutations.
\begin{theorem}
\label{t:twoper}
$~$
\begin{itemize}
\item[(a)]
For any $\sigma_1,\sigma_2 \in \ccs_n$, the permutation complex $X(\sigma_1,\sigma_2)$ is either contractible or is homotopy equivalent to a wedge of spheres of dimensions at least $1$.
\item[(b)]
For any $1 \leq k_1,\ldots,k_m$ there exist $\sigma_1,\sigma_2 \in \ccs_n$ such that
\begin{equation}
\label{e:disp}
X(\sigma_1,\sigma_2) \simeq S^{k_1} \vee \cdots \vee S^{k_m}.
\end{equation}
\end{itemize}
\end{theorem}

Our third result concerns the homotopy type of the complex of injective words associated to an arbitrary simplicial complex.
\begin{definition}
\label{d:ciw}
Let $K$ be a simplicial complex with vertex set $V$. The set of injective words associated to $K$ is
$$M(K)= \bigcup_{\emptyset \neq \sigma \in K} \inj(\sigma) \subset \inj(V).$$
The \emph{complex of injective words} on $K$ is $\Gamma(K)=\Delta(M(K))$.
\end{definition}
\noindent
Let $\lk(K,\sigma)=\{\tau \in K: \tau\cap\sigma=\emptyset,\;\text{and}\; \tau\cup\sigma\in K\}$
denote the link of a simplex $\sigma \in K$.
\begin{theorem}
\label{t:hty}
Let $K$ be a finite connected simplicial complex. Then there is a homotopy equivalence
\begin{equation}
\label{e:hty}
\Gamma(K) \simeq \bigvee_{\sigma\in K}\bigvee_{D(|\sigma|)}\Sigma^{|\sigma|}\lk(K,\sigma).
\end{equation}
\end{theorem}

Finally, we discuss some simple probabilistic aspects of permutation complexes.
\begin{definition}
\label{Def:xnm}
For positive integers $m$ and $n$,
let $\mathcal{X}_{m,n}$ denote the probability space of permutation complexes $X(\sigma_1,\ldots,\sigma_m)$, where $\sigma_1,\ldots,\sigma_m$ are drawn independently from the uniform probability space $\ccs_n$. Write $X_{m,n}$ for a random complex in $\mathcal{X}_{m,n}$.
\end{definition}

\begin{proposition}
\label{p:echar}
Let $m,n \geq 1$. Then
\begin{itemize}
\item[(a)]
The expectation of the reduced Euler characteristic of complexes in $\mathcal{X}_{m,n}$ satisfies
\begin{equation}
\label{e:echar}
E[\tilde{\chi}(X_{m,n})]=
\sum_{k=0}^n (-1)^{k-1} \left(1-\left(1-\frac{1}{k!}\right)^m\right)\binom{n}{k}k!.
\end{equation}
\item[(b)]
\begin{equation}
\label{e:prb2}
|E[\tilde{\chi}(X_{2,n})]|=O(n^{-1/4}).
\end{equation}
\end{itemize}
\end{proposition}

The paper is organized as follows. In Section \ref{Sec:uniform} we prove Theorem \ref{t:uniform}
 on the realizability of stable homotopy types by permutation complexes. In Section \ref{Sec:two} we establish Theorem \ref{t:twoper} on the homotopy type of complexes generated by two permutations. Section \ref{s:inj} is concerned with complexes of injective words and includes
a proof of Theorem \ref{t:hty} and a simple application to a result of Randal-Williams and Wahl \cite{RW-W17}. In Section \ref{s:random} we prove Proposition \ref{p:echar} and discuss some additional probabilistic questions concerning permutation complexes. We conclude in Section \ref{s:con} with a number of remarks and open problems.

\section{Realizability of Stable Homotopy Types}
\label{Sec:uniform}
In this section we show that the stable homotopy type of a simplicial complex can be realized
by a permutation complex. We need some preliminary notions.

\begin{definition}\label{Def:Q}
For $\sigma_1,\ldots,\sigma_m \in \ccs_n$, let $Q(\sigma_1,\ldots,\sigma_m)$ denote the poset on $[n]$ with the order $\prec$ given by
$i \prec j$ if and only if $\sigma_k^{-1}(i)<\sigma_k^{-1}(j)$ for all $1 \leq k \leq m$.
\end{definition}
\noindent
We first note the following
\begin{lemma}
\label{Lem:Intersection=sd(Q)}
\label{c:xsd}
For any permutations $\sigma_1,\ldots, \sigma_m \in \ccs_n$,
\begin{equation}
\label{e:xsd}
\bigcap_{k=1}^m X(\sigma_k) \cong \sd \Delta(Q(\sigma_1,\ldots,\sigma_m)).
\end{equation}
\end{lemma}
\begin{proof} Let $w=j_0\cdots j_p \in \inj([n])$. Then $w$ is a vertex of
$\bigcap_{k=1}^m X(\sigma_k)$ if and only if $\sigma_k^{-1}(j_0)<\cdots<\sigma_k^{-1}(j_p)$
for all $1 \leq k \leq m$, and that holds if and only if $\{j_0,\ldots,j_p\}$ is a simplex in $\Delta(Q(\sigma_1,\ldots,\sigma_m))$, i.e. a vertex of
$\sd \Delta(Q(\sigma_1,\ldots,\sigma_m))$. Defining a map $\varphi$ on the vertex set of
$\bigcap_{k=1}^m X(\sigma_k)$ by $\varphi(j_0\cdots j_p)= \{j_0,\ldots,j_p\}$, it is straightforward to check that $\varphi$ is a simplicial isomorphism between $\bigcap_{k=1}^m X(\sigma_k)$
and $\sd \Delta(Q(\sigma_1,\ldots,\sigma_m))$.
\end{proof}
\noindent
The {\it concatenation} of two permutations $\tau \in \ccs_n, \eta \in \ccs_{\ell}$, is the permutation $\tau \odot \eta \in \ccs_{n+{\ell}}$ given by
$$
\tau \odot \eta (j)=\left\{
\begin{array}{ll}
\tau(j) & 1 \leq j \leq n, \\
n+\eta(j-n) &  n+1 \leq j \leq n+\ell.
\end{array}
\right.~~
$$
Let $\tau_1,\ldots,\tau_k \in \ccs_n$ and $\eta_1,\ldots,\eta_k \in \ccs_{\ell}$. Let $A*B$ denote the simplicial join of complexes $A$ and $B$. Clearly
\begin{equation}
\label{e:jointe}
\Delta(Q(\tau_1 \odot \eta_1,\ldots,\tau_m \odot \eta_m))\cong
\Delta(Q(\tau_1,\ldots,\tau_m))*
\Delta(Q(\eta_1,\ldots,\eta_m)).
\end{equation}
We will also need the following
\begin{lemma}
\label{c:union}
Let $X_1,\ldots,X_d$ be simplicial complexes such that
$\bigcap_{i \in I} X_i$ is contractible for all $\emptyset \neq I \subsetneqq [d]$.
Then
\begin{equation}
\label{e:union}
\bigcup_{i=1}^d X_i \simeq \Sigma^{d-1}\left(\bigcap_{i=1}^d X_i\right).
\end{equation}
\end{lemma}
\begin{proof}
We argue by induction on $d$. The induction basis $d=2$ is straightforward,
see e.g. Corollary 7.4.3 in \cite{Brown}. Assume now that $d>2$. Applying the induction hypothesis to the family
$\{X_i \bigcap X_d\}_{i=1}^{d-1}$, it follows that
\begin{equation}
\label{e:unid}
\left(\bigcup_{i=1}^{d-1} X_i\right) \bigcap X_d = \bigcup_{i=1}^{d-1} (X_i \bigcap X_d)
\simeq \Sigma^{d-2} \left(\bigcap_{i=1}^{d-1} (X_i \bigcap X_d)\right)=
\Sigma^{d-2} \left(\bigcap_{i=1}^{d} X_i \right).
\end{equation}
\noindent
Next note that the intersection $\bigcap_{i \in I} X_i$ is contractible for any subset
$\emptyset \neq I \subset [d-1]$, hence $\bigcup_{i=1}^{d-1} X_i$ is contractible by the nerve lemma.
Applying the induction basis to
the pair $\bigcup_{i=1}^{d-1} X_i$ and $X_d$, and using (\ref{e:unid}) we obtain
\begin{equation}
\label{e:unidd}
\begin{split}
&\bigcup_{i=1}^{d} X_i = \left(\bigcup_{i=1}^{d-1} X_i\right) \cup X_d
\simeq \Sigma \left( \left(\bigcup_{i=1}^{d-1} X_i\right) \bigcap X_d \right) \\
&\simeq \Sigma \left(\Sigma^{d-2} \left(\bigcap_{i=1}^{d} X_i \right) \right)
\simeq \Sigma^{d-1} \left(\bigcap_{i=1}^{d} X_i \right).
\end{split}
\end{equation}
\end{proof}

\begin{definition}
\label{Def:order_dimension}
The \emph{order dimension} $\dim(P)$ of a finite partially ordered set $(P,\prec)$ is the minimal number $k$ of linear orders $\prec_1,\ldots,\prec_k$ on the elements of $P$, such that for any $x,y \in P$, it holds that $x\prec y$ if and only if $x \prec_i y$ for all $1\le i\le k$.
\end{definition}
\noindent
Dushnik and Miller \cite{DM41} observed that $\dim(P) \leq |P|$. Their bound was improved by
Hiraguchi \cite{H51} to $\dim(P) \leq \lfloor |P|/2 \rfloor$ for posets of size $|P| \geq 4$.
\noindent
The main result of this section is the following detailed version of Theorem \ref{t:uniform}.
\begin{theorem}
\label{Thm:uniform}
Let $Y$ be a simplicial complex and let $|P(Y)|=n$. Then there exist $d=\dim(P(Y))$ permutations
$\sigma_1,\ldots,\sigma_d \in \ccs_{n+2d}$ such that
\begin{equation}
\label{e:uni}
X(\sigma_1,\ldots,\sigma_d)\simeq \Sigma^{2d-1} Y.
\end{equation}
\end{theorem}
\begin{proof}
Let $P(Y)=\{F_1,\ldots,F_n\}$.
Choose $d$ permutations $\tau_1,\ldots,\tau_d \in S_n$ such that $F_i \subset F_j$ if and only if $\tau_k^{-1}(i) \leq \tau_k^{-1}(j)$ for all $1 \leq k \leq d$. Then
\begin{equation}
\label{e:sdyd}
\sd Y \cong \Delta\left(Q(\tau_1,\ldots,\tau_d)\right).
\end{equation}
\noindent
For $1 \leq k \leq d$ let $\eta_k \in S_{2d}$ denote
the transposition that switches $2k-1$ and $2k$, and let $\sigma_k=\tau_k \odot \eta_k$.
Observe that if $\emptyset \neq I \subsetneqq [d]$ and if
$j\in [d] \setminus I$, then any element in the poset $Q\left(\eta_i: i \in I\right)$ is comparable to the element $2j-1$ (and also to $2j$). Hence
\begin{equation}
\label{e:contr}
\Delta\big(Q\left(\eta_i: i \in I\right)\big) \simeq \emph{point}.
\end{equation}
On the other hand
\begin{equation}
\label{e:sphere}
\Delta\big(Q\left(\eta_1,\ldots,\eta_d\right)\big)=
\{1,2\}*\{3,4\}* \cdots * \{2d-1,2d\}\cong
 S^{d-1}.
\end{equation}
\noindent
Using (\ref{e:xsd}), (\ref{e:jointe}), (\ref{e:sdyd}), (\ref{e:contr}) and (\ref{e:sphere}), it follows that for any $\emptyset \neq I \subset \{1,\ldots,d\}$
\begin{equation}
\label{e:inter}
\begin{split}
\bigcap_{i \in I} X(\sigma_i) &\cong \sd \Delta\big(Q(\sigma_i:i \in I)\big) \simeq \Delta\big(Q(\tau_i \odot \eta_i:i \in I)\big) \\
&\simeq \Delta\big(Q(\tau_i:i \in I)\big)*\Delta\big(Q(\eta_i:i \in I)\big) \\
&\simeq
\begin{cases}
\emph{point} & I \neq \{1,\ldots,d\}, \\
Y*S^{d-1} & I=\{1,\ldots,d\},
\end{cases}
\\
&\simeq
\begin{cases}
\emph{point} & I \neq \{1,\ldots,d\}, \\
\Sigma^{d} Y & I=\{1,\ldots,d\}.
\end{cases}
\end{split}
\end{equation}
Applying Lemma \ref{c:union} to the family $\{X(\sigma_i)\}_{i=1}^d$, we thus obtain
\begin{equation*}
\label{e:htype}
\begin{split}
X(\sigma_1,\ldots,\sigma_d)&= \bigcup_{k=1}^d X(\sigma_k)
\simeq \Sigma^{d-1}\bigcap_{k=1}^d X(\sigma_k)  \simeq \Sigma^{2d-1}Y.
\end{split}
\end{equation*}
\end{proof}

\begin{example}\label{Ex:RP2}
 Let $Y$ be the $6$ point triangulation of $\Rea \PPR^2$. The number of nonempty faces of $Y$ is $n=31$. As $d \leq \lfloor n/2 \rfloor =15$, it follows that $\Sigma^{29} \Rea \PPR^2$ is homotopy equivalent to a permutation complex.
\end{example}

Recently Dejan Govc discovered an example of a permutation complex generated by 15 permutations in $\ccs_{4}$:
\[\{1234, 1243, 1324, 1432, 2143, 2314, 2341, 2413, 3124, 3142, 3214, 3421, 4123, 4231, 4312\}\]
 that realizes the homotopy type of $\Sigma \Rea\PPR^2$. Thus it seems reasonable to conjecture that there may be more economical ways to stably realize homotopy types than what is claimed in Theorem \ref{Thm:uniform}.



\section{Complexes Generated by Two Permutations}
\label{Sec:two}
In this section we study the homotopy types of permutation complexes $X(\sigma_1,\sigma_2)$, for $\sigma_1,\sigma_2 \in \ccs_n$. As $X(\sigma_1), X(\sigma_1)$ are both contractible, it follows by Lemma \ref{c:union} that
\begin{equation}
\label{e:unioni}
X(\sigma_1,\sigma_2)=X(\sigma_1) \cup X(\sigma_2) \simeq
\Sigma \big(X(\sigma_1) \cap X(\sigma_2)\big).
\end{equation}
Hence it suffices to analyse the intersections $X(\sigma_1) \cap X(\sigma_2)$.
It will be useful to consider the following, essentially equivalent, setup.

\begin{definition}
\label{2-perm}
Let $(R,\prec)$ be a finite linearly ordered set and let $\phi:R \rightarrow \Rea$ be an injective
function. Let $Y(R,\phi)$ be the order complex of the poset $(R,\prec')$ where
$r \prec' r'$ if both $r \prec r'$ and $\phi(r) < \phi(r')$.
A triple $(R,\prec, \phi)$ will be referred to as an \emph{intersection triple}.
\end{definition}

For example, if $(R, \prec)$ is the set $[n]$ with its natural order and $\sigma \in \ccs_{n}$ is a permutation, viewed as an injective function from $[n]$ to $\Rea$, then
$Y([n], \phi)=X(\text{id}_n) \cap X(\sigma)$, where $\text{id}_n$ denotes the identity permutation in $\ccs_n$.  We now determine the possible homotopy types of $Y(R,\phi)$.

\begin{lemma}
\label{c:mv}
Let $(R,\prec, \phi)$ be an intersection triple, as in Definition \ref{2-perm}, and fix some $r \in R$. Let $R'=R-\{r\}$ and let $\phi'$ denote the restriction of $\phi$ to $R'$.
Let
$$
R''=\{t \in R' \; : \; t \prec r~~ \text{and} ~~\phi(t)<\phi(r)\} \cup
\{t \in R' \; : \; t \succ r ~~\text{and}~~ \phi(t)>\phi(r)\}
$$
and let $\phi''$ denote the restriction of $\phi$ to $R''$.
Then
\begin{equation}
\label{e:de}
\begin{split}
Y(R,\phi)&=Y(R',\phi') \cup \bigl(\{r\}*Y(R'',\phi'')\bigr) \\
 Y(R'',\phi'')&=Y(R',\phi') \cap \bigl(\{r\}*Y(R'',\phi'')\bigr).
\end{split}
\end{equation}
\noindent
Furthermore, there is an exact sequence
\begin{equation}
\label{e:mv}
\cdots \rightarrow \thh_k(Y(R'',\phi'')) \rightarrow \thh_k(Y(R',\phi')) \rightarrow \thh_k(Y(R,\phi))
\rightarrow \thh_{k-1}(Y(R'',\phi'')) \rightarrow  \cdots
\end{equation}
\end{lemma}
\noindent
\begin{proof} The decomposition (\ref{e:de}) is straightforward.  The exact sequence (\ref{e:mv}) follows from (\ref{e:de})
and the Mayer-Vietoris theorem, since $\{r\}*Y(R'',\phi'')$ is a cone.
\end{proof}
\begin{definition}
\label{d:alternating}
Let $(R,\prec,\phi)$ be an intersection triple. A \emph{$k$-alternating sequence in $(R,\phi)$} is an increasing chain $r_1\prec \cdots \prec r_{2k} \in R$ such that $\phi(r_{2i}) < \phi (r_{2i-1})$ for all $1 \leq i \leq k$ and $\phi(r_{2i-1}) < \phi(r_{2i+2})$ for all $1 \leq i \leq k-1$. Let $\alpha(R,\phi)$ denote the maximal $k$ such that $R$ contains a $k$-alternating sequence.
\end{definition}
\noindent
For a space $Y$, let $\hd(Y)=\max\{k: \thh_k(Y) \neq 0\}$ if $Y$ is not $\Int$-acyclic, and $\hd(Y)=-1$ otherwise.
The main result of this section is the following extended version of Theorem \ref{t:twoper}.
Parts (a) and (b) of the two statements are equivalent
by Eq. (\ref{e:unioni}).
\begin{theorem}
\label{thm:hom}
$~$
\begin{itemize}
\item[(a)]
Let $(R,\prec, \phi)$ be an intersection triple. Then $Y(R,\phi)$ is either contractible or is homotopy equivalent to a wedge of spheres.
\item[(b)]
for any $k_1,\ldots,k_m \geq 0$ there exist $R$ and $\phi$ such that
$$
Y(R,\phi)\simeq S^{k_1} \vee \cdots \vee S^{k_m}.
$$
\item[(c)]
$\alpha(R,\phi) \geq \hd(Y(R,\phi))+1$.
\end{itemize}
\end{theorem}
\begin{proof}
Parts (a) and (c) will be proved simultaneously by induction on $|R|$. The case $|R|=1$ is clear.
Assume that $|R| \geq 2$ and let $r_1,r_2 \in R$ be the elements with the smallest, respectively second smallest, value of $\phi$, i.e.
$$
\phi(r_1)=\min\{\phi(t)\; : \; t \in R \}
$$
and
$$
\phi(r_2)=\min\{\phi(t)\;|\; r_1 \neq t \in R \}.
$$
Using the notation of Lemma \ref{c:mv}, we consider two cases:

\noindent {\bf (i)} $r_1 \prec r_2$:
Set $r=r_2$. Then
$$
R'=R-\{r_2\}
$$
and
$$
R''=\{r_1\} \cup \{t\in R \;|\; t \succ r_2\}.
$$
Clearly, $Y(R'',\phi'')$ is a cone on $r_1$ and is therefore contractible. By the decomposition (\ref{e:de}) above, there is a homotopy equivalence
\begin{equation}
\label{e:case1}
Y(R,\phi) \simeq Y(R',\phi').
\end{equation}
Thus (a) follows directly from the induction hypothesis, while for (c) we also use the monotonicity
$\alpha(R,\phi) \geq \alpha(R',\phi')$.

\medskip
\noindent {\bf (ii)} $r_2\prec r_1$: Set $r=r_1$.
Then
$$
R'=R-\{r_1\}
$$
and
$$
R''=\{t\in R \;|\; t \succ r_1\}.
$$
Note that $\{r_2\}*Y(R'',\phi'') \subset Y(R',\phi')$ and therefore $Y(R'',\phi'')$ is
contractible in $Y(R',\phi')$.
Again, by (\ref{e:de}) there is a homotopy equivalence
\begin{equation}
\label{e:case2}
Y(R,\phi) \simeq Y(R',\phi') \vee \Sigma Y(R'',\phi''),
\end{equation}
and (a) follows by induction. To show (c), observe that if $t_1\prec \cdots \prec t_{2k}$ is a $k$-alternating
sequence in $(R'',\phi'')$, then $r_2 \prec r_1 \prec t_1 \prec \cdots \prec t_{2k}$ is a $(k+1)$-alternating
sequence in $(R,\phi)$. Hence
\begin{equation}
\label{e:asa}
\alpha(R,\phi) \geq \alpha(R'',\phi'')+1.
\end{equation}
Combining (\ref{e:asa}) with the induction hypothesis and with (\ref{e:case2}), we obtain
\begin{equation*}
\label{e:asa1}
\begin{split}
\alpha(R,\phi) &\geq \max\{\alpha(R',\phi'),\alpha(R'',\phi'')+1\} \\
&\geq \max\{ \hd\left(Y(R',\phi')\right)+1,\hd\left(Y(R'',\phi'')\right)+2\}  \\
&\geq \max\{ \hd\left(Y(R',\phi')\right),\hd\left(\Sigma Y(R'',\phi'')\right)\}+1  \\
&=\hd \left(Y(R,\phi)\right)+1.
\end{split}
\end{equation*}
\noindent
We next prove part (b), showing that every finite wedge of spheres $S^{k_1} \vee \cdots \vee S^{k_m}$ is homotopy equivalent to a space of type $Y(R, \phi)$. We argue by induction on $m$. For $m=1$, $k_1=k \geq 0$, let $R=[2k+2]$ and let
$\phi$ denote the permutation given by $\phi(2i-1)=2i, \phi(2i)=2i-1$ for $1 \leq i \leq k$.
Then $Y(R,\phi)$ is the octahedral $k$-sphere. Suppose now that $m \geq 2$. We may assume that
$k_1,\ldots,k_{m-1} \geq k_m$. By induction there exists a linearly ordered set $(R_1,\prec_1)$
and an injective function $\phi_1:R_1\rightarrow \Rea$ such that
$Y(R_1,\phi_1) \simeq S^{k_1-k_m} \vee \cdots \vee S^{k_{m-1}-k_m}$. We consider two cases:

\noindent
(1) $k_m=0$. Let $t$ be a new element $t \not\in R_1$ and let $R=R_1 \cup \{t\}$. Let $\prec$ be the linear order on $R$ given by $r_1 \prec r_1' \prec t$ if $r_1 \prec_1 r_1' \in R_1$. Define $\phi: R \rightarrow \Rea$ by $\phi(r_1)=\phi_1(r_1)$ for $r_1 \in R_1$ and $\phi(t)<
\min\{\phi_1(r_1):r_1 \in R_1\}$. Then
$$
Y(R,\phi)=Y(R_1,\phi_1) \cup \{t\} \simeq S^{k_1} \vee \cdots \vee S^{k_{m-1}} \vee S^0 =
S^{k_1} \vee \cdots \vee S^{k_{m}}.
$$
\noindent
(2) $k_m \geq 1$. By the induction basis $m=1$ there exists an $(R_2,\prec_2)$, and an injective function
$\phi_2:R_2 \rightarrow \Rea$, such that $Y(R_2,\phi_2)\simeq S^{k_{m}-1}$.
We may assume that
$R_1 \cap R_2=\emptyset$ and $\phi_1(r_1)<\phi_2(r_2)$ for all $(r_1,r_2) \in R_1 \times R_2$.
Let $t$ be a new element and let $R= R_1 \cup \{t\} \cup R_2$. Define a linear order $\prec$ on $R$ by $r_1 \prec t \prec r_2$ for all $(r_1,r_2) \in R_1 \times R_2$,
and $r_i \prec r_i'$ if and only if $r_i \prec_i r_i'$ for $(r_i,r_i') \in R_i \times R_i$ and
$i \in \{1,2\}$.
Define $\phi:R \rightarrow \Rea$ by $\phi(r_i)=\phi_i(r_i)$ for $r_i \in R_i$ and
$\phi(t) < \min_{r_1 \in R_1}\phi_1(r_1)$. Writing $R'=R_1 \cup R_2$ and letting $\phi'$ be the restriction of $\phi$ to $R'$, it is clear that $Y(R',\phi') \simeq Y(R_1,\phi_1)*Y(R_2,\phi_2)$.
By case (ii) of part (a) above we have
\begin{equation}
\label{e:wedges}
\begin{split}
Y(R,\phi) &\simeq Y(R',\phi') \vee \Sigma Y(R_2,\phi_2) \\
&\simeq\big(Y(R_1,\phi_1)*Y(R_2,\phi_2)\big) \vee \Sigma Y(R_2,\phi_2)  \\
&\simeq \left(\left(S^{k_1-k_m} \vee \cdots \vee S^{k_{m-1}-k_m}\right)*S^{k_m-1}\right)
\vee \Sigma S^{k_m-1} \\
&\simeq \left(S^{k_1} \vee \cdots \vee S^{k_{m-1}}\right) \vee S^{k_m} \\
&\simeq S^{k_1} \vee \cdots \vee S^{k_{m}}.
\end{split}
\end{equation}
\end{proof}

\begin{example}
For $S=[n]$ and $\phi(i)=a_i$,  write
$Y(a_1,\ldots,a_n)=Y(S,\phi)$.

\begin{equation}
\label{e:example}
\begin{split}
Y(3,5,8,1,7,2,4,6) &\simeq Y(3,5,8,1,7,4,6) \\
&\simeq Y(3,5,8,7,4,6) \vee \Sigma Y(7,4,6) \\
&\simeq Y(3,5,8,7,6) \vee \Sigma Y(7,4) \\
&\simeq \Sigma Y(7,4) \simeq S^1.
\end{split}
\end{equation}
The first and third equivalences are case (i), while the second is case (ii) in the proof of Theorem \ref{thm:hom}.
\end{example}

\section{The Homotopy Type of $\Gamma(K)$}
\label{s:inj}

In this section we prove Theorem \ref{t:hty} on the homotopy type of the complex of injective words $\Gamma(K)$ associated with a connected simplicial complex $K$. Our main tool is a powerful homotopy decomposition theorem due to
Bj\"{o}rner, Wachs and Welker \cite{BWW05}. We first recall some definitions. Let $(P,<_P)$, $(Q,<_Q)$ be two posets. A map $f:P \rightarrow Q$ is a {\it poset map} if
$x \leq_P y$ implies $f(x) \leq_Q f(y)$.
For an element $q \in Q$
let $Q_{<q}=\{x \in Q: x <_Q q\}$. The subsets $Q_{\leq q}$ and $Q_{> q}$ are defined similarly.
The {\it length} $\ell(Q)$ of $Q$ is the number of elements in a maximal chain of $Q$ minus $1$,
i.e., $\ell(Q)=\dim \Delta(Q)$.
\begin{theorem}[\cite{BWW05}]
\label{t:bww}
Let $f:P \rightarrow Q$ be a poset map such that $\Delta(Q)$ is connected and for all $q \in Q$ the fiber $\Delta(f^{-1}(Q_{\leq q}))$ is $\ell(f^{-1}(Q_{<q}))$-connected.
Then
\begin{equation}
\label{e:bww}
\Delta(P)\simeq \Delta(Q) \vee \bigvee_{q \in Q} \Delta\left(f^{-1}(Q_{\leq q})\right)* \Delta(Q_{>q}).
\end{equation}
\end{theorem}
\noindent
{\it Proof of Theorem \ref{t:hty}.}
Recall that $K$ is a finite simplicial complex and $(P(K),\prec)$ is the poset of nonempty faces of
$K$ ordered by inclusion.
The map $g:P(K)_{\succ \sigma} \rightarrow P(\lk(K,\sigma))$ given by
$g(\tau)=\tau \setminus \sigma$ is a poset isomorphism. It follows that
\begin{equation}
\label{e:sdlink}
\begin{split}
\Delta\left(P(K)_{\succ \sigma}\right) &\cong \Delta\left(P(\lk(K,\sigma))\right) \\
&=\sd \lk(K,\sigma) \simeq \lk(K,\sigma).
\end{split}
\end{equation}
\noindent
Next consider the poset map $f:M(K) \rightarrow P(K)$ given by
$f(v_0 \cdots v_k)=\{v_0,\ldots,v_k\}$.
Observe that if $\sigma \in P(K)$ then
$f^{-1}\big(P(K)_{\preccurlyeq {\sigma}}\big)=\inj(\sigma)$. Hence, by Theorem \ref{t:bw}
\begin{equation}
\label{e:invimg}
\Delta\Big(f^{-1}\big(P(K)_{\preccurlyeq \sigma}\big)\Big)
\cong \Delta\left(\inj(\sigma)\right)=\Gamma(\sigma) \simeq \bigvee_{D(|\sigma|)} S^{|\sigma|-1}.
\end{equation}
Therefore $\Delta\Big(f^{-1}\big((P(K))_{\prec {\sigma}}\big)\Big)$
is $(|\sigma|-2)$-connected. As $\ell\Big(f^{-1}\big(P(K)_{\prec {\sigma}}\big)\Big)=|\sigma|-2$, it follows that the poset map $f$ satisfies the conditions of Theorem \ref{t:bww}. Using (\ref{e:bww}) and (\ref{e:invimg}) we obtain
\begin{equation}
\label{e:bww1}
\begin{split}
\Gamma(K)&=\Delta(M(K))\simeq \Delta(P(K)) \vee \bigvee_{\emptyset \neq \sigma \in X}
\Delta\Big(f^{-1}\big((P(K))_{\preccurlyeq \sigma}\big)\Big)*
\Delta\big((P(K))_{\succ \sigma}\big) \\
&\simeq K \vee \bigvee_{\emptyset \neq \sigma \in K} \Gamma(\sigma)*\lk(K,\sigma) \simeq
K \vee \bigvee_{\emptyset \neq \sigma \in K}
\bigvee_{D(|\sigma|)} S^{|\sigma|-1}*\lk(K,\sigma) \\
&\simeq \bigvee_{\sigma \in K} \bigvee_{D(|\sigma|)}\Sigma^{|\sigma|}\lk(K,\sigma).
\end{split}
\end{equation}
\enp
\noindent
As an immediate corollary we obtain the homological version of Theorem \ref{t:hty}.
\begin{corollary}
\label{c:homg}
Let $K$ be a connected simplicial complex and let $A$ be an abelian group. Then for $r \geq 0$
\begin{equation}
\label{e:ahr}
\thh_r\big(\Gamma(K); A\big) \cong \bigoplus_{\sigma \in K} \thh_{r-|\sigma|}\big(\lk(K,\sigma); A\big)^{\oplus D(|\sigma|)}.
\end{equation}
\end{corollary}
\noindent
Let $p_{\FF}(K,t)=\sum_{i \geq 0} \tilde{\beta}_{i-1}(K;\FF) t^i$ denote the reduced Poincar\'{e} polynomial of $K$ over a field $\FF$. Corollary \ref{c:homg} implies the following
\begin{corollary}
\label{c:poing}
If $K$ is connected then
\begin{equation}
\label{e:poing}
p_{\FF}\left(\Gamma(K),t\right)=\sum_{\sigma \in K} D(|\sigma|) t^{|\sigma|} p_{\FF}(\lk(K,\sigma),t).
\end{equation}
\end{corollary}
\noindent
We conclude this section with a simple application of Theorem \ref{t:hty}.
\begin{definition}
A simplicial complex $K$ is \emph{weakly Cohen-Macaulay} of dimension $n$ if
$\lk(K,\sigma)$ is $(n-\dim \sigma -2)$-connected for all $\sigma \in K$.
\end{definition}
\noindent
Theorem \ref{t:hty} provides a very simple proof of the following result of
Randal-Williams and Wahl (Proposition 2.14 in \cite{RW-W17}).
\begin{proposition}[\cite{RW-W17}]
\label{p:rww}
If $K$ is weakly Cohen-Macaulay of dimension $n$, then $\Gamma(K)$ is $(n-1)$-connected.
\end{proposition}
\begin{proof}
Let $\sigma \in K$. By assumption $\lk(K,\sigma)$ is $(n-|\sigma|-1)$-connected,
and hence $\Sigma^{|\sigma|}\lk(K,\sigma)$ is $(n-1)$-connected. As this holds for all $\sigma \in K$, it follows from (\ref{e:hty}) that $\Gamma(K)$ is $(n-1)$-connected.
\end{proof}

\section{Random Permutation Complexes}
\label{s:random}
\noindent
{\it Proof of Proposition \ref{p:echar}.}
(a) Let $w=i_1 \cdots i_{k}$ be a fixed injective word of length $k$ on $[n]$.
The probability that $w$ is a subword of a random permutation $\sigma \in S_n$
is clearly $1/k!$. Hence
\begin{equation}
\label{e:prw}
\pr \left[ w \in W(\sigma_1,\ldots,\sigma_m)\right]=1-\left(1-\frac{1}{k!}\right)^m.
\end{equation}
Let
$a_{k-1}(\sigma_1,\ldots,\sigma_m)=\rk\, A_{k-1}(\sigma_1,\ldots,\sigma_m)$ denote the number of words of length $k$ in $W(\sigma_1,\ldots,\sigma_m)$. Theorem \ref{t:acom} implies that
\begin{equation}
\label{e:eulerc}
\tilde{\chi}\big(X(\sigma_1,\ldots,\sigma_m)\big)=-1+\sum_{k=1}^{n} (-1)^{k-1} a_{k-1}(\sigma_1,\ldots,\sigma_m).
\end{equation}
Using (\ref{e:eulerc})
and (\ref{e:prw}) we obtain
$$
E[\tilde{\chi}(X_{m,n})]=-1+\sum_{k=1}^{n} (-1)^{k-1} E[a_{k-1}] =
 \sum_{k=0}^n (-1)^{k-1} \left(1-\left(1-\frac{1}{k!}\right)^m\right)\binom{n}{k}k!.
$$
\noindent
(b) The $n$-th Laguerre polynomial
is $L_n(x)=\sum_{k=0}^n \frac{(-1)^k}{k!} \binom{n}{k} x^k$.
By the asymptotic formula for Laguerre polynomials (see e.g. Theorem 7.6.4 in \cite{Szego67}) we obtain
$$
\left|E[\tilde{\chi}(X_{2,n})]\right|= \left|\sum_{k=0}^n \frac{(-1)^{k}}{k!} \binom{n}{k}\right|=\left|L_n(1)\right|=O(n^{-\frac{1}{4}}).
$$
{\enp}
\noindent
{\bf Remarks:}
\\
1. Let $n$ be fixed and $m \rightarrow \infty$.
Then
$$
\lim_{m \rightarrow \infty} \pr[X_{m,n}=\Gamma(\dn)]=1.
$$
\noindent
It follows that
$$\lim_{m \rightarrow \infty} E[\tilde{\chi}(X_{n,m})] =
\tilde{\chi}(\Gamma(\dn))=(-1)^{n-1}D(n).$$
This of course agrees with (\ref{e:echar}) as indeed
\begin{equation}
\label{e:limit}
\begin{split}
\lim_{m \rightarrow \infty}  E[\tilde{\chi}(X_{n,m})] &=
\lim_{m \rightarrow \infty}
\sum_{k=0}^n (-1)^{k-1} \left(1-(1-\frac{1}{k!})^m\right)\binom{n}{k}k! \\
&=\sum_{k=0}^n (-1)^{k-1} \binom{n}{k} k!=(-1)^{n-1} D(n).
\end{split}
\end{equation}
\\
2. For $\sigma \in \ccs_n$ let $\lambda_n(\sigma)$ denote the maximal length of an increasing subsequence in $\sigma$. The random variable $\lambda_n$ has been studied in depth for the last 50 years, see Aldous and Diaconis \cite{AD99} for a detailed survey. In particular, the asymptotics of $E[\lambda_n]$ is given by the following celebrated result of Logan and Shepp \cite{LS77} and Kerov and Vershik \cite{VK77}.
\begin{theorem}[\cite{LS77,VK77}]
\label{t:lskv}
$E[\lambda_n] \sim 2 \sqrt{n}$ as $n \rightarrow \infty$.
\end{theorem}
\noindent
Clearly
$$\dim \left( X(\sigma_1) \cap  X(\sigma_2) \right) = \lambda_n(\sigma_2^{-1} \sigma_1)-1,$$
and therefore
$$
\hd\left(X(\sigma_1,\sigma_2)\right) \leq \hd\left(X(\sigma_1) \cap X(\sigma_2)\right)+1 \leq \lambda_n(\sigma_2^{-1} \sigma_1).
$$
Hence
\[
\limsup_{n \rightarrow \infty} E[\hd(X_{2,n})/\sqrt{n}] \leq 2.
\]
A somewhat better upper bound is given in the following
\begin{proposition}
\label{p:lims}
\[
\limsup_{n \rightarrow \infty} E\left[\hd(X_{2,n})/\sqrt{n}\right] \leq \frac{e}{2}.
\]
\end{proposition}
\begin{proof}
Let $\alpha_n$ be the random variable on the symmetric group $\ccs_n$ given by
$\alpha_n(\sigma)=\alpha([n],\sigma)$, i.e. the maximal $k$ such that $\sigma$ contains a $k$-alternating sequence.
Fix $c>e/2$ and let $k=c\sqrt{n}$.
Let $\gamma_{n,k}(\sigma)$ denote the number of $k$-alternating sequences in $\sigma$.
Then by Stirling's formula
\begin{equation*}
\label{e:ehalfa}
\pr[\alpha_n \geq k] \leq E[\gamma_{n,k}]
= \binom{n}{2k}\frac{1}{(2k)!} \leq \frac{1}{\sqrt{n}}\left(\frac{e}{2c}\right)^{4k}.
\end{equation*}
Theorem \ref{thm:hom}(c) implies that
$\hd\left(X(\sigma_1,\sigma_2)\right) \leq \alpha_n(\sigma_2^{-1} \sigma_1)$.
Hence
\begin{equation*}
\label{e:ehalfb}
\begin{split}
\frac{E\left[\hd(X_{2,n})\right]}{\sqrt{n}} &\leq \frac{E\left[\alpha_n\right]}{\sqrt{n}}
\leq \frac{1}{\sqrt{n}}
\left(k\cdot \pr[\alpha_n \leq k]+\frac{n}{2}\cdot\pr[\alpha_n \geq k]\right) \\
&\leq \frac{k}{\sqrt{n}}+\frac{1}{2} \left(\frac{e}{2c}\right)^{4c\sqrt{n}} \rightarrow c.
\end{split}
\end{equation*}
\end{proof}
\noindent
We end this section with recursive formulas that can be used for numerical computation of the expectation of the betti numbers of $X_{2,n}$.

\begin{definition}\label{Def:bnk_unk_vnrk}
Let $b_{n,k}$ be the random variable on the probability space $\ccs_n$
given by $b_{n,k}(\pi)=\tilb_k(Y([n],\pi))$.
Let $u(n,k)=E[b_{n,k}]$ and for $1 \leq r \leq n$ let
$v(n,r,k)=E[b_{n,k}|\pi^{-1}(1)=r]$.
\end{definition}
\noindent
Clearly $u(n,k)=\frac{1}{n} \sum_{r=1}^n v(n,r,k)$.
The proof of Theorem \ref{thm:hom} implies the following recursion for $v(n,r,k)$.
\begin{proposition}
\label{p:expb}
$~$
\begin{itemize}
\item[(a)]
$v(1,1,0)=0$.
\item[(b)]
For $1<n$ and $k \geq 0$:
\begin{equation}
\label{e:nr}
v(n,n,k)=\delta_{k,0}+\frac{1}{n-1} \sum_{r=1}^{n-1} v(n-1,r,k).
\end{equation}
\item[(c)]
For $1 \leq r<n$ and $k \geq 0$:
\begin{equation}
\label{e:rln}
\begin{split}
v(n,r,k)&=\frac{n-r}{n-1} v(n-1,r,k)+\frac{1}{n-1}\sum_{j=1}^{r-1} v(n-1,j,k) \\
&+\frac{(r-1)(1-\delta_{k,0})}{(n-1)(n-r)} \sum_{i=1}^{n-r}v(n-r,i,k-1).
\end{split}
\end{equation}
\end{itemize}
\end{proposition}
\begin{proof}
Part (a) is clear. To prove part (b) note that if $\pi^{-1}(1)=n$ then
$Y([n],\pi)$ is the disjoint union of $Y([n-1],\pi_{|[n-1]})$ and the point $\{n\}$,
where $\pi_{|[n-1]}$ denotes the restriction of $\pi$ to $[n-1]$.
It follows that
\begin{equation}
\label{e:vnrkb}
\begin{split}
v(n,n,k)&=E[b_{n,k}|\pi^{-1}(1)=n]=\delta_{k,0} +E[b_{n-1,k}] \\
&=\delta_{k,0}+\frac{1}{n-1} \sum_{r=1}^{n-1} v(n-1,r,k).
\end{split}
\end{equation}
\noindent
For part (c) fix $j \neq r <n$ and consider two cases.

\noindent {\bf (i)}: $j>r$. If $\pi \in \ccs_n$ satisfies
$\pi^{-1}(1)=r$, $\pi^{-1}(2)=j$, then by (\ref{e:case1}):
\begin{equation*}
\label{e:case1a}
\tilde{\beta}_k\left(Y([n],\pi)\right)=
\tilde{\beta}_k\left(Y([n]\setminus \{j\},\pi_{|[n]\setminus \{j\}})\right),
\end{equation*}
and hence
\begin{equation}
\label{e:case1aa}
E[b_{n,k}|\pi^{-1}(2)=j ~\wedge~ \pi^{-1}(1)=r]=v(n-1,r,k).
\end{equation}

\noindent {\bf (ii)}: $j<r$. If $\pi \in \ccs_n$ satisfies
$\pi^{-1}(1)=r$, $\pi^{-1}(2)=j$, then by (\ref{e:case2}):
\begin{equation*}
\label{e:case2a}
\tilde{\beta}_k\left(Y([n],\pi)\right)=
\tilde{\beta}_k\left(Y([n]\setminus\{r\},\pi_{|[n]\setminus\{r\}})\right)
+(1-\delta_{k,0})\tilde{\beta}_{k-1}\left(Y([n]\setminus [r],\pi_{|[n]\setminus [r]})\right),
\end{equation*}
and hence
\begin{equation}
\label{e:case2aa}
E[b_{n,k}|\pi^{-1}(2)=j ~\wedge~ \pi^{-1}(1)=r]=
v(n-1,j,k)+
\frac{1-\delta_{k,0}}{n-r} \sum_{i=1}^{n-r}v(n-r,i,k-1).
\end{equation}
Averaging $E[b_{n,k}|\pi^{-1}(2)=j ~\wedge~ \pi^{-1}(1)=r]$ over all $j \neq r$,
using (\ref{e:case1aa}) and (\ref{e:case2aa}), we obtain (\ref{e:rln}).
\end{proof}

\section{Concluding Remarks}
\label{s:con}
In this paper we studied some combinatorial and topological aspects of complexes of injective words. Our main results concern the realizability of stable homotopy types by permutation complexes and the existence of explicit homotopy decompositions of the complex of injective words associated with a simplicial complex. Our work suggests some questions regarding complexes of injective words.
\begin{itemize}
\item
In Theorem \ref{t:uniform} we showed that every stable homotopy type can be realized
by a permutation complex. It is natural to ask whether iterated suspensions are essential for such realization. For example, is the real projective plane homotopy equivalent to some $X(\sigma_1,\ldots,\sigma_m)$?
\item
In Proposition \ref{p:lims} we proved that $E\left[\hd(X_{2,n})/\sqrt{n}\right] \leq \frac{e}{2}$.
The constant $\frac{e}{2}$ can be slightly improved, but
in view of Theorem \ref{t:lskv} and of some numerical evidence, we suggest the following
\begin{conjecture}
\label{c:hdim}
There exists a constant $0<\gamma<1$ such that $E[\hd(X_{2,n})] \sim \gamma \sqrt{n}$
as $n \rightarrow \infty$.
\end{conjecture}
\item
By Theorem \ref{t:twoper}, if $\sigma_1,\sigma_2 \in \ccs_n$
then $X(\sigma_1,\sigma_2)$ is homotopic to a wedge of spheres
$S^{k_1} \vee \cdots \vee S^{k_m}$, and hence $\tilde{\chi}(X(\sigma_1,\sigma_2))=\sum_{i=1}^m (-1)^{k_i}$. On the other hand, Proposition \ref{p:echar} implies that
$|E[\tilde{\chi}(X_{2,n})]|=o(1)$. This is of course consistent with the intuitive guess that
for a random $X_{2,n} \in \mathcal{X}_{2,n}$, the decomposition (\ref{e:disp}) should contain about the same number of odd and even spheres.
\item
Our paper deals only with injective words. For recent work on complexes associated with words with letter repetitions, see Kozlov's paper \cite{Kozlov19}.
\end{itemize}

\ \\ \\
{\bf Conflict of Interest Statement}
\\
On behalf of all authors, the corresponding author states that there is no conflict of interest.

\end{document}